\documentclass{amsart}
\usepackage{graphicx}
\usepackage{verbatim}
\usepackage[arrow,matrix]{xy}
\title{Topologically $4$-chromatic graphs and signatures of odd cycles}
\author[G. Simons, C. Tardif, D. Wehlau]{Gord Simons, Claude Tardif, 
David Wehlau}\thanks{The third author's research is supported by grants 
from NSERC and ARP}
\address{Royal Military College of Canada \\ PO Box 17000 Stn Forces,
Kingston, ON\\ Canada, K7K 7B4}

\newcommand{\cqfd}{\begin{flushright}\rule{8pt}{9pt}\end{flushright} \par}
\newtheorem{define}{Definition}

\newtheorem{theorem}[define]{Theorem}
\newtheorem{lemma}[define]{Lemma}

\newcommand{\bd}{\begin{define} \rm}
\newcommand{\ed}{\end{define}}

\newcommand{\ind}{\mathrm{ind}}
\newcommand{\coind}{\mathrm{coind}}
\newcommand{\hoco}[1]{\mathrm{Hom}(K_2,#1)}
\newcommand{\boco}[1]{\mathrm{B}(#1)}

\newcommand{\soi}[1]{\mbox{\bf \rm 1}^{<}_{#1}}

\newcommand{\basu}{S}
\newcommand{\boucle}{L}

\newfont{\Bb}{msbm10 scaled\magstep1}

\begin{document}

\begin{abstract}
We investigate group-theoretic ``signatures'' of odd cycles of a graph,
and their connections to topological obstructions to 3-colourability.
In the case of signatures derived from free groups, we prove that the existence
of an odd cycle with trivial signature is equivalent to having the coindex
of the hom-complex at least 2 (which implies that the chromatic number is
at least 4). In the case of signatures derived from elementary abelian 2-groups
we prove that the existence of an odd cycle with trivial signature 
%
%is equivalent to having the index of the hom-complex at least 2 
is a sufficient condition for having the index of the hom-complex at least 2
(which again implies that the 
chromatic number is at least 4).
\end{abstract} 

\maketitle
{\small \noindent{\bf Keywords:}
Graph colourings, homomorphisms, Hom complexes, free groups
\newline \noindent {\em AMS 2010 Subject Classification:}
05C15.}
\smallskip

\section{Introduction} This paper is motivated by so-called
``topological bounds'' on the chromatic number of a graph:
\begin{equation}\label{topbounds}
\chi(H) \geq \ind(\hoco{H}) + 2 \geq \coind(\hoco{H}) + 2.
\end{equation}
Here, $\hoco{H}$ is a ``hom-complex'' which can be viewed both as a 
$\mathbb{Z}_2$-poset and as the geometric realisation of its order complex. 
Its index $\ind(\hoco{H})$ and coindex $\coind(\hoco{H})$ will defined 
in the next section.

The bounds (\ref{topbounds}) have been useful in determining chromatic numbers
for various classes of graphs. However for general graphs, the index
and coindex are not known to be computable. In contrast, the chromatic
number is in NP. The computational aspects of such topological invariants
are now being investigated (see \cite{CKMSVW,CKV}). In this paper 
we focus on the case when the bounds give a chromatic number
of at least $4$. We present an algebraic approach.

For a graph $H$, let $A(H)$ denote the set of its arcs. That is,
for $[u,v] \in E(H)$, $A(H)$ contains the two arcs $(u,v)$ and $(v,u)$.
Let $\mathcal{V}$ be a variety of groups (in the sense of universal algebra:
a class of groups defined by a set of identities). 
Let $\mathcal{F}_{\mathcal{V}}(A(H))$ 
be the free group in $\mathcal{V}$ generated by the elements of $A(H)$.  
We define the congruence $\theta$ on $\mathcal{F}_{\mathcal{V}}(A(H))$  
by the relations
$$
\begin{array}{rcl}
(a,b) (c,b)^{-1} (c,d) (a,d)^{-1} & \theta & 1
\end{array}
$$
for all $4$-cycles $a, b, c, d$ of $H$.
The group $\mathcal{G}_{\mathcal{V}}(H)$ is defined as the
quotient $\mathcal{F}_{\mathcal{V}}(A(H)) / \theta$.
Let $C_{n}$ denote the cycle with vertex-set 
$\mathbb{Z}_{n} = \{0, \ldots, n-1\}$ and edges $[i,i+1]$, $i \in \mathbb{Z}_{n}$.
If $n$ is odd and $f: C_n \rightarrow H$ is a homomorphism, we define the 
{\em $\mathcal{V}$-signature} $\sigma_{\mathcal{V}}(f)$ of $f$ as
$$
\sigma_{\mathcal{V}}(f) 
= \prod_{i = 0}^{n-1} (f(2i),f(2i+1)) \cdot (f(2i+2),f(2i+1))^{-1},
$$
where the indices are taken modulo $2n+1$ and the product
is developed left to right: $\prod_{i = 0}^{k}x_i = x_0 x_1 \cdots x_k$ 
rather than $x_k x_{k-1} \cdots x_0$.

We will consider two varieties of groups and related signatures:
the variety $\mathcal{V}_1$ of all groups and the variety
$\mathcal{V}_2$ of elementary abelian 2-groups (where we use additive notation).
We let $\sigma_1(f)$ and $\sigma_2(f)$ denote respectively
$\sigma_{\mathcal{V}_1}(f)$ and $\sigma_{\mathcal{V}_2}(f)$.
We prove the following results.
\begin{theorem} \label{coindth}
Let $H$ be a graph. Then $\coind(\hoco{H}) \geq 2$ if and only if for some odd $n$,
there exists a homomorphism $f: C_n \rightarrow H$ such that $\sigma_1(f) = 1$.
\end{theorem}
\begin{theorem} \label{indth}
Let $H$ be a graph. If for some odd $n$ there exists a homomorphism 
$f: C_n \rightarrow H$ such that $\sigma_2(f) = 0$, then $\ind(\hoco{H}) \geq 2$.
\end{theorem}
We will show that the existence of an odd $n$ and a homomorphism 
$f: C_n \rightarrow H$ such that $\sigma_2(f) = 0$ can be
decided in polynomial time. Therefore if the converse of 
Theorem~\ref{indth} holds, then the question as to whether
a graph $H$ satisfies $\ind(\hoco{H}) \geq 2$ can be decided in
polynomial time. In contrast, Theorem~\ref{coindth} provides a necessary and 
sufficient algebraic condition for a graph $H$ to satisfy  $\coind(\hoco{H}) \geq 2$,
but it is not clear whether this condition can be decided at all, let alone in 
polynomial time. In fact, as pointed out by Zimmerman~\cite{zimmerman}, it is not clear
whether the word problem in $\mathcal{G}_{\mathcal{V}_1}(H)$
is always decidable.

\section{Preliminaries}
In this section we introduce the terminology necessary
to define the index and coindex of hom-complexes, and
characterise them in ways that will allow us to prove
Theorems~\ref{coindth} and \ref{indth}.

\subsection{Topology}
The {\em hom-complex $\hoco{H}$ of $H$} is the set with elements
$(A,B)$ such that $A, B$ are nonempty subsets of $H$ and every element
of $A$ is joined by an edge of $H$ to every element of $B$. Here, $K_2$
denotes the complete graph on two vertices $0$ and $1$; the name
``hom''-complex is derived from the fact that if $(A,B) \in \hoco{H}$,
then for any $a \in A$ and $b \in B$, there is a homomorphism $f$ of
$K_2$ to $H$ defined by $f(0) = a$ and $f(1) = b$.

We view $\hoco{H}$ primarily as a $\mathbb{Z}_2$-poset, that is a poset
with a fixed-point free automorphism of order two (denoted $-$).
The order relation on $\hoco{H}$ is coordinatewise inclusion, 
and the $\mathbb{Z}_2$ involution is given by $-(A,B) = (B,A)$.
A {\em $\mathbb{Z}_2$-map} between $\mathbb{Z}_2$-posets $P$ and $Q$ 
is an order-preserving map $f: P \rightarrow Q$ such that
$f(-x) = -f(x)$.

Any poset $P$ can also be viewed as a simplicial complex, by viewing
chains as simplices. The {\em geometric realization} of $P$
is the topological space $\Delta P \subseteq \mathbb{R}^P$
induced by the functions $f: P \rightarrow [0,1]$ whose support
(i.e., set of elements with nonzero image) is a chain in $P$,
and which satisfy $\sum \{ f(x) : x \in P \} = 1$. 
If $P$ is a $\mathbb{Z}_2$-poset, then $\Delta P$ is a
$\mathbb{Z}_2$-space, that is, a topological space with a
fixed-point free homeomorphism of order 2. A {\em $\mathbb{Z}_2$-map} 
between $\mathbb{Z}_2$-spaces $X$ and $Y$ 
is a continous map $f: X \rightarrow Y$ such that
$f(-x) = -f(x)$.

The index $\ind(X)$ and coindex $\coind(X)$ of a 
$\mathbb{Z}_2$-space $X$ are defined in terms
of the unit sphere $S_n \subseteq \mathbb{R}^{n+1}$, 
viewed as a $\mathbb{Z}_2$-space:
\begin{itemize}
\item $\ind(X)$ is the smallest $n$ such that $X$ admits a
$\mathbb{Z}_2$-map to $S_n$;
\item $\coind(X)$ is the largest $n$ such that $S_n$ admits a
$\mathbb{Z}_2$-map to $X$.
\end{itemize}
The fact that $\ind(S_n) = \coind(S_n) = n$ is not trivial, but is 
a restatement of the Borsuk-Ulam theorem.

For a $\mathbb{Z}_2$-poset $P$, we write $\ind(P)$ and $\coind(P)$ respectively
for $\ind(\Delta P)$ and $\coind(\Delta P)$. 
Now consider the $(2n+2)$-element $\mathbb{Z}_2$-poset $Q_n$, with elements 
$\{ \pm 0, \allowbreak \ldots, \allowbreak \pm n\}$ ordered by the relation 
$\{+i, -i\} < \{+j, -j\}$ (in $Q_n$) when $i < j$ (in $\mathbb{N}$).
Then $Q_n$ is the face-poset of the cross-polytope of dimension $n$,
and therefore $\Delta Q_n$ is $\mathbb{Z}_2$-homeomorphic to $S_n$. This
correspondence can be used to characterise the index and the coindex
of a $\mathbb{Z}_2$-poset in terms of order-preserving $\mathbb{Z}_2$-maps.

The {\em barycentric subdivision} of a poset $P$ is the poset $\basu(P)$
whose elements are the chains of $P$, ordered by inclusion.
Note that when $P$ is a $\mathbb{Z}_2$-poset, $\basu(P)$ is also
a $\mathbb{Z}_2$-poset. The exponential notation is used to denote
iterated barycentric subdivisions. By simplicial approximation, the
following holds for any $\mathbb{Z}_2$-poset $P$:
\begin{itemize}
\item $\ind(P)$ is the smallest $n$ such that for some $m$, $\basu^m(P)$ 
admits a $\mathbb{Z}_2$-map to $Q_n$;
\item $\coind(P)$ is the largest $n$ such that for some $m$, $\basu^m(Q_n)$ 
admits a $\mathbb{Z}_2$-map to $P$.
\end{itemize}

\subsection{Graph theory}
For a graph $H$, the characterisation of $\ind(\hoco{H})$ given just above
will be sufficient to prove Theorem~\ref{indth} in Section~\ref{prindth} below.
For a proof of Theorem~\ref{coindth}, we rely on a further characterisation
of $\coind(\hoco{H})$ in terms of graph homomorphisms.

The {\em categorical product} of two graphs $G$ and $G'$ 
is the graph $G\times G'$ defined by 
\begin{eqnarray*}
V(G\times G') & = & V(G) \times V(G'), \\
E(G \times G') & = & \{ [(u,u'),(v,v')] : [u,v]  \in E(G) \mbox{ and } [u',v'] \in E(G') \}.
\end{eqnarray*}
For $q \in \mbox{\Bb N}^*$, let $\mbox{\Bb P}_q$ denote the path with vertices 
$0, 1, \ldots, q$ linked consecutively, with a loop at $0$. 
For a graph $G$, the $q$-th cone $M_q(G)$
(or $q$-th generalised Mycielskian) over $G$ is the graph 
$(G \times \mbox{\Bb P}_q)/ \sim_q$,
where $\sim_q$ is the equivalence which identifies all vertices whose second coordinate is $q$.
The vertex $(V(G) \times \{q\})/ \sim_q$ is called the {\em apex} of $M_q(G)$,
while $V(G) \times \{0\}$ is the {\em base} of $M_q(G)$. Any set
$V(G) \times \{i\}$ is called a {\em level} of $M_q(G)$.

\begin{figure}[htp]
\centering
\includegraphics{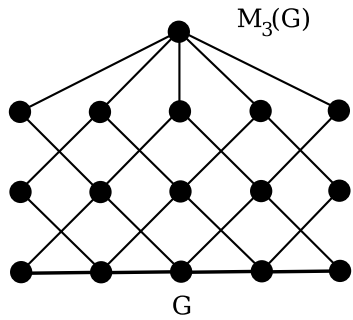}
\caption{Generalised Mycielskian}  
\label{mycielskian}
\end{figure}
 
The cone construction allows us to define classes of ``generalised Mycielski graphs''
inductively: Let $\mathcal{K}_2 = \{ K_2 \}$, and for $k \geq 3$, put
$$
\mathcal{K}_k = \{ M_q(G) : G \in \mathcal{M}_{k-1}, q \in \mbox{\Bb N}^*\}.
$$
\begin{lemma}[\cite{STW}] \label{coind=mycielski}
For any graph $H$, $\coind(\hoco{H})$ is the largest $k$ such that there exist
a $G \in \mathcal{K}_{k+2}$ admitting a homomorphism to $H$. 
\end{lemma}

The complex used in~\cite{STW} was the box complex $\boco{H}$ rather than the hom-complex
$\hoco{H}$. However the two complexes are $\mathbb{Z}_2$-homotopy equivalent
by a result of Csorba~\cite{csorba}.

\section{Proof of Theorem \ref{coindth}}  \label{prcoindth}

\subsection{Overview of the proof} One implication of
Theorem \ref{coindth} has been proved in~\cite{STW}:

\begin{lemma}[Proposition 5 of ~\cite{STW}]
If $\coind(\hoco{H}) \geq 2$, then there exists an odd cycle $C_n$
and a homomorphism $f: C_n \rightarrow H$ such that $\sigma_1(f) = 1$.
\end{lemma} 

\begin{proof} For reference it is worthwhile to give a sketch of 
the proof here. In view of Lemma~\ref{coind=mycielski}, our
hypothesis implies the existence of a homomorphism 
$g: M_q(C_n) \rightarrow H$. We will write $u_{i,j}$ for $g((i,j)/ \sim_q)$. 
(Note that $(i,j)/ \sim_q = \{(i,j)\}$ except when $j = q$; $u_{i,q}$ 
is the image of the apex of  $M_q(C_n)$ for any $i \in \mathbb{Z}_n$.)
We will show that the homomorphism $f: C_n \rightarrow H$ defined by 
$f(i) = u_{i,0}$ satisfies $\sigma_1(f) = 1$.

Consider the expressions $L_j, R_j, j = 0, \ldots, q-1$ in 
$\mathcal{G}_{\mathcal{V}_1}(H)$ given by
\begin{eqnarray*}
L_j & = & (u_{0,j},u_{1,j+1})(u_{2,j},u_{1,j+1})^{-1}
(u_{2,j},u_{3,j+1})(u_{4,j},u_{3,j+1})^{-1} \cdots (u_{n-1,j},u_{0,j+1}), \\
R_j & = & (u_{0,j+1},u_{1,j})^{-1}(u_{2,j+1},u_{1,j})
(u_{2,j+1},u_{3,j})^{-1}(u_{4,j},u_{3,j+1}) \cdots (u_{n-1,j},u_{0,j+1}).
\end{eqnarray*}
We then have 
$$L_j R_j = \prod_{i=0}^{n-1}(u_{2i,j},u_{2i+1,j+1})(u_{2i+2,j},u_{2i+1,j+1})^{-1}, 
j = 0, \ldots, q-1.$$ 
In particular, $L_{q-1} R_{q-1}$ simplifies to $1$, 
since  $u_{2i+1,q}$ is the constant image of the apex of $M_q(C_n)$. Also, by definition
of $\theta$, for $j = 1, \ldots, q-1$, we have
$$(u_{2i,j},u_{2i+1,j+1})(u_{2i+2,j},u_{2i+1,j+1})^{-1} 
= (u_{2i,j},u_{2i+1,j-1})(u_{2i+2,j},u_{2i+1,j-1})^{-1}$$
for $i = 0, \ldots, n-1$.
Therefore 
$$L_j R_j = \prod_{i=0}^{n-1}(u_{2i,j},u_{2i+1,j-1})(u_{2i+2,j},u_{2i+1,j-1})^{-1} 
j = 0, \ldots, q-1.$$ 
This is the image $\phi(R_{j-1} L_{j-1})$ of $R_{j-1} L_{j-1}$ under the 
(well-defined) automorphism $\phi$ of $\mathcal{G}_{\mathcal{V}_1}(H)$
which interchanges $(x,y)$ with $(x,y)^{-1}$ for all $(x,y) \in A(H)$.
Therefore if $L_j R_j = 1$, then $R_{j-1} L_{j-1} = 1$ and
$L_{j-1} R_{j-1} = 1$, since $L_{j-1} R_{j-1}$
and $R_{j-1} L_{j-1}$ are conjugates. Therefore, $L_j R_j= 1$ for all
$j = 0, \ldots, q-1$. Again by definition of $\theta$, we then have
$$\begin{array}{lcccr}
\sigma_1(f) & = & \prod_{i=0}^{n-1}(u_{2i,0},u_{2i+1,0})(u_{2i+2,0},u_{2i+1,0})^{-1} & & \\
& = & \prod_{i=0}^{n-1}(u_{2i,0},u_{2i+1,1})(u_{2i+2,0},u_{2i+1,1})^{-1} & = & L_0 R_0 = 1.
\end{array}
$$\end{proof}

It would be nice to be able to reverse the arguments of this proof
to prove the second direction. That is, start with 
$f: C_n \rightarrow H$ such that $\sigma_1(f) = 1$, and use 
the definition of $\theta$ to extend it to a homomorphism 
$g: M_q(C_n) \rightarrow H$. However an example in~\cite{STW} shows that
this is not always possible; $C_n$ may be to small as a base.

Thus we need to start from the basic information provided by the equation 
$\sigma_1(f) = 1$. Since $\mathcal{G}_{\mathcal{V}_1}(H) = 
\mathcal{F}_{\mathcal{V}}(A(H)) / \theta$ this means that in
$\mathcal{F}_{\mathcal{V}_1}$, $\sigma_1(f)$ is equal to a product
$\prod_{i=1}^{k} \gamma_i^{-1} \rho_i \gamma_i$ of conjugates
of the relations $\rho_i$ defining $\theta$.

The first difficulty here is that there is a useful feature common 
to $\sigma_1(f)$ and the generators $\rho_i$ of $\theta$, 
which is lost in the expression 
$\prod_{i=1}^{k} \gamma_i \rho_i \gamma_i^{-1}$: the fact that words alternate in
forward arcs and inverses of backward arcs along a walk in $H$.
This feature provides a natural connection between algebraic expressions and path
homomorphisms. For this reason, we will fix a root vertex $r$ in $H$,
 and associate to each arc $(x,y)$ a closed walk 
from $r$ through $(x,y)$, called a ``loop''. The use of loops will transform 
$\prod_{i=1}^{k} \gamma_i \rho_i \gamma_i^{-1}$ into a word which is much longer, 
but which has the desired alternating property. We will identify this word
with an expression of the type $L_j R_j$ defining homomorphic
images of two consecutive levels of some $M_q(C_{2m+1})$.

We will then extend this homomorphism to the apex of $M_q(C_{2m+1})$
using the definition of $\theta$, and towards its base using the
simplification of $\prod_{i=1}^{k} \gamma_i \rho_i \gamma_i^{-1}$ to $\sigma_1(f)$.
Our basic tool to convert algebraic simplifications to homomorphism
extensions is the extension along ``bricks'', that is, essentially
rectangular pieces that dissect  $M_q(C_{2m+1})$.

The last phase of the extension will be the connection to the
base of $M_q(C_{2m+1})$, which is equivalent to finding an extension
that is equal on two consecutive levels.
 
\subsection{Loops} Most of our work will be done
in the free monoid $(A(H) \cup A(H)^{-1})^*$ 
generated by $A(H) \cup A(H)^{-1}$, where 
$A(H)^{-1} = \{ (u,v)^{-1} : (u,v) \in A(H) \}$
is a set of symbols disjoint from $A(H)$.
Of course, $\mathcal{F}_{\mathcal{V}_1}(A(H)) = (A(H) \cup A(H)^{-1})^*/\iota$,
where $\iota$ is the congruence which identifies
$(u,v)(u,v)^{-1}$ and $(u,v)^{-1}(u,v)$ to $1$ for all $(u,v) \in A(H)$.
However, for a suitable correspondence between words and walks, it is
sometimes useful to avoid this identification.

A walk $u_0, u_1, \ldots, u_n$ in $H$ is the image of a homomorphism 
$f$ of some path with vertices $0, 1, \ldots n$ linked consecutively.
To such a walk we can associate a word 
$$\omega(f) = (u_0,u_1)(u_2,u_1)^{-1}(u_2,u_3)(u_4,u_3)^{-1} \cdots$$
ending in $(u_{n-1},u_n)$ or $(u_n,u_{n-1})^{-1}$ depending on whether
$n$ is odd or even. This word  alternates
symbols from $A(H)$ and symbols from $A(H)^{-1}$, with the symbol
following $(u_{2i},u_{2i+1})$ being $(u_{2i+2},u_{2i+1})^{-1}$ for some 
$u_{2i+2} \in V(H)$, and the symbol following $(u_{2i},u_{2i-1})^{-1}$ being 
$(u_{2i},u_{2i+1})$ for some $u_{2i+1} \in V(H)$.
Conversely, a word with these properties naturally corresponds to
a walk in $H$.

Now for our purposes we can assume that $H$ is connected and nonbipartite, if necessary by 
restricting our attention to the component of $H$ that contains an odd cycle with
trivial signature. We fix a root vertex $r$ and for every vertex $u$ of $H$,
we fix an even path $p_e(u)$ and an odd path $p_o(u)$, both from $r$ to $u$. 
For an arc $(u,v)$ of $H$, we define the loop $\boucle(u,v) \in (A(H) \cup A(H)^{-1})^*$
corresponding to $(u,v)$ by
$$\boucle(u,v) = \omega(p_e(u))\cdot(u,v)\cdot \omega(p_o(v))^{-1}.$$
Of course, $(a_1 a_2 \cdots a_n)^{-1}$ means $a_n^{-1} a_{n-1}^{-1}\cdots a_1^{-1}$,
though this needs to be stated formally since inversion does not exist in 
$(A(H) \cup A(H)^{-1})^*$. With this notation,
we define $\boucle((u,v)^{-1})$ as $(\boucle(u,v))^{-1}$.

The loop function naturally extends to $(A(H) \cup A(H)^{-1})^*$
by putting $\boucle(a_1 \cdots a_n) = \boucle(a_1) \cdots \boucle(a_n)$.
The map $\boucle: (A(H) \cup A(H)^{-1})^* \rightarrow (A(H) \cup A(H)^{-1})^*$
is an endomorphism whose image consists of words corresponding to
even closed walks rooted at $r$. 
%% Also, since 
%% $\boucle((u,v)^{-1})/\iota = (\boucle((u,v)/\iota)^{-1}$, $\boucle$ can be 
%% quotiented to $\mathcal{F}_{\mathcal{V}_1}(A(H))$
%% as to make the following diagram commute.
%% $$
%%  \xymatrix{
%%    (A(H) \cup A(H)^{-1})^* \ar[r]^{\boucle} \ar[d]^{\iota}      
%%& (A(H) \cup A(H)^{-1})^* \ar[d]^{t_H} \\
%%    \mathcal{F}_{\mathcal{V}_1}(A(H)) \ar[r]^{\boucle}                   
%%& \mathcal{F}_{\mathcal{V}_1}(A(H)) 
%%  }
$$$$

\subsection{Cycles with trivial signature}
Now let $C_{n'}$ be an odd cycle and $f: C_{n'} \rightarrow H$ 
a homomorphism such that $\sigma_1(f) = 1$.
For our purposes, it is useful to assume that for 
$i \in \mathbb{Z}_{n'} = V(C_{n'})$, we have $f(i-1) \neq f(i+1)$.
This can be done without loss of generality, since
if $f(i-1) = f(i+1)$, then we can remove $i$ and identify $i-1$ and $i+1$ 
to create a copy of $C_{n'-2}$ on
which $f$ induces a homomorphism $f': C_{n'-2} \rightarrow H$. We then have
$\sigma_1(f') = \sigma_1(f)$, since 
$\sigma_1(f')$ is obtained from $\sigma_1(f)$ by cancelling out 
$(f(i-1),f(i))$ with $(f(i+1),f(i))^{-1}$
and $(f(i),f(i-1))^{-1}$ with $(f(i),f(i+1))$
(unless $i = 0$, in which case we get a conjugation of the identity
element instead of the second cancellation).

We begin by lifting the condition $\sigma_1(f) = 1$ from $\mathcal{G}(H)$ to
$\mathcal{F}_{\mathcal{V}_1}(A(H))$: by definition, this means that
there exists a sequence 
$$\{ \rho_i = (a_i,b_i)(c_i,b_i)^{-1}(c_i,d_i)(a_i,d_i)^{-1} : i = 1, \ldots, k \}$$
of generators of $\theta$ and a sequence
$\gamma_1, \ldots, \gamma_k$ of conjugating elements such that 
\begin{equation} \label{sigequation}
\prod_{i=0}^{n'-1} (f(2i),f(2i+1))\cdot (f(2i+2),f(2i+1))^{-1}
= \prod_{i=1}^{k} \gamma_i \cdot \rho_i \cdot \gamma_i^{-1}.
\end{equation}
Note that the left side is reduced in $\mathcal{F}_{\mathcal{V}_1}(A(H))$, 
since $f(j-1) \neq f(j+1)$ for all $j \in \mathbb{Z}_{n'}$. This
means that the right side simplifies to the left, by repeatedly
cancelling out adjacent terms that are inverse of each other.

We view the terms in this equation as elements of $(A(H) \cup A(H)^{-1})^*$
(that is, as their simplest preimage under $\iota$).
In this way the terms $\boucle(\rho_i)$ and $\boucle(\gamma_i)$ are well defined.
Now consider the word
$$w' = \prod_{i=1}^{k} \boucle(\gamma_i)\cdot \boucle(\rho_i) \cdot \boucle(\gamma_i^{-1}).$$
Its length is a multiple of $4$, since for each $i$, the length
of $\boucle(\gamma_i)\cdot \boucle(\rho_i) \cdot \boucle(\gamma_i^{-1})$ is a multiple of $4$.
Therefore for any neighbour $s$ of $r$, the word 
$w = w' \cdot (r,s) \cdot (r,s)^{-1}$ has length equal to
$2n$ for some odd $n$. We then have $w = \omega(g)$ for some closed walk 
$g$. We identify the domain of $g$ with the $2n$-cycle $C_n \times K_2$, 
whose vertices correspond to two consecutive levels of some $M_q(C_n)$. 
We will show  such that it is possible to choose $q$ and the correspondence
such that $g$ extends to a homorphism of $M_q(C_n)$ to $H$.

%%Note that this task is consistent algebraically. 
%%In $\mathcal{F}_{\mathcal{V}_1}(A(H))$,
%%the word $w$ simplifies to
%%$$\prod_{i=1}^{k} \boucle(\gamma_i)\cdot \omega(p_e(a_i)) \cdot \rho_i 
%%\cdot \omega(p_e(a_i))^{-1}\cdot \boucle(\gamma_i)^{-1},$$
%%which quotients to the identity in $\mathcal{G}(H)$.
%%Also in $\mathcal{F}_{\mathcal{V}_1}(A(H))$, the simplification of
%%$\prod_{i=1}^{k} \gamma_i \cdot \rho_i \cdot \gamma_i^{-1}$ to 
%%$\sigma_1(f')$ is done by recursively cancelling consecutive
%%terms that are inverse of each other. Since $\boucle(a^{-1}) = \boucle(a)^{-1}$,
%%we can use the same simplifications to reduce $w$ to $\boucle(\sigma_1(f'))$,
%%which further reduces to 
%%$\omega(p_e(f'(0)) \cdot \sigma_1(f) \cdot \omega(p_e(f'(0))^{-1}$.
%%Our homomorphism extension is essentially an adaptation of these
%%simplifications.
 
\subsection{Bricks and simplifications}
For integers $i$ and $j$, the bricks 
$B_{\mbox{s}}(i,j)$ and $B_{\mbox{d}}(i,j)(i,j)$ are the 
graphs defined by
\begin{eqnarray*}
V(B_{\mbox{s}}(i,j)) & = & \{ (x,y) : 0 \leq x \leq i, 0 \leq y \leq j,
\mbox{ $x$ and $y$ have the same parity} \}, \\
E(B_{\mbox{s}}(i,j)) & = & \{ [(x,y),(x',y')] : |x - x'| = 1, |y - y'| = 1 \};\\
V(B_{\mbox{d}}(i,j)) & = & \{ (x,y) : 0 \leq x \leq i, 0 \leq y \leq j,
\mbox{ $x$ and $y$ have different parities} \}, \\
E(B_{\mbox{d}}(i,j)) & = & \{ [(x,y),(x',y')] : |x - x'| = 1, |y - y'| = 1 \};\\
\end{eqnarray*}
Thus, $B_{\mbox{s}}(i,j)$ and $B_{\mbox{d}}(i,j)$ are the two connected components of 
the categorical product
of paths of lengths $i$ and $j$.
The sets of vertices of $B_{\mbox{s}}(i,j)$ and $B_{\mbox{d}}(i,j)$ 
with second coordinate $0$ are called the 
{\em lower side} of $B_{\mbox{s}}(i,j)$ and $B_{\mbox{d}}(i,j)$, and similarly, 
their {\em upper}, {\em left}, and {\em right} sides are defined 
by obvious conditions. 

For any $i$, $B_{\mbox{s}}(2i,1)$ and $B_{\mbox{d}}(2i,1)$ are paths of length $2i$.
We use the following extension properties of their homomorphisms to $H$.
\begin{lemma} \label{extension}
Let $h: B_{\mbox{s}}(2i,1) \rightarrow V(H)$, $h': B_{\mbox{s}}(2i,1) \rightarrow V(H)$ 
be homomorphisms.
\begin{itemize}
\item[(i)] For any $j\geq 0$, the map 
$\hat{h}$ defined by $\hat{h}(x,y) = h(x,(y\mod 2))$ 
is a homomorphism of $B_{\mbox{s}}(2i,j)$ to $H$.
Similarly, for any $j\geq 0$, the map 
$\hat{h}'$ defined by $\hat{h}'(x,y) = h(x,1-(y\mod 2))$ 
is a homomorphism of $B_{\mbox{d}}(2i,j)$ to $H$.
\item[(ii)] If $h(x, x \mod 2) = h(2i - x, x \mod 2)$ for $x = 0, \ldots i$, then
there exists a homomorphism  
$\hat{h}: B_{\mbox{s}}(2i,2\lceil i/2 \rceil) \rightarrow H$
extending $h$ such that $\hat{h}$ is identically equal to $h(0)$ on the left, 
upper and right sides of $B_{\mbox{s}}(2i,2\lceil i/2 \rceil)$.
Similarly, if $h'(x, 1-(x \mod 2)) = h'(2i - x, 1-(x \mod 2))$ for $x = 0, \ldots i$,
then there exists a homomorphism  
$\hat{h}': B_{\mbox{d}}(2i,2\lfloor i/2 \rfloor + 1) \rightarrow H$
extending $h$ such that $\hat{h}$ is identically equal to $h(0,1)$ on the left, 
upper and right sides of $B_{\mbox{d}}(2i,2\lfloor i/2 \rfloor +1)$.
\end{itemize}
\end{lemma}
\begin{proof} Item (i) is straightforward.
To prove the first part of item (ii) , we note that the distance between 
two vertices $(x,y)$, $(x',y')$ of $B_{\mbox{s}}(2i,j)$ is the ``bus distance'' 
$\max \{ |x-x'|, |y-y'| \}$.
Therefore the map $\hat{h}: B_{\mbox{s}}(2i,2\lceil i/2 \rceil) \rightarrow H$ defined by 
$\hat{h}(x,y) = h(x',y')$, where $(x',y') \in B_{\mbox{s}}(2i,1)$ is at the same distance as 
$(x,y)$ from $(i,(i\mod 2))$, is a homomorphism with the prescribed properties.
The second part is proved similarly. 
\end{proof}

\subsection{Extension to the apex}

We apply Lemma~\ref{extension} to extend the homomorphism
$g: C_{n}\times K_2 \rightarrow H$ towards the apex of some $M_q(C_n)$.
Consider the restrictions $g_1, \ldots, g_k$ of $g$ 
such that $\omega(g_i) = \boucle(\gamma_i)\cdot \boucle(\rho_i) \cdot \boucle(\gamma_i^{-1})$.
For a fixed $i$, we further decompose $g_i$ into five restrictions
$h_a, h_b, h_c, h_d, h'_a$ such that
\begin{eqnarray*}
\omega(h_a) & = & \boucle(\gamma_i)\cdot \omega(p_e(a_i)), \\
\omega(h_b) & = & \omega(p_o(b_i))^{-1} \circ \omega(p_o(b_i)), \\
\omega(h_c) & = & \omega(p_e(c_i))^{-1} \circ \omega(p_e(c_i)), \\
\omega(h_d) & = & \omega(p_o(d_i))^{-1} \circ \omega(p_o(d_i)), \\
\omega(h'_a) & = & \omega(p_e(a_i))^{-1} \cdot \boucle(\gamma_i).
\end{eqnarray*}

\begin{figure}[h]
\centering
\includegraphics{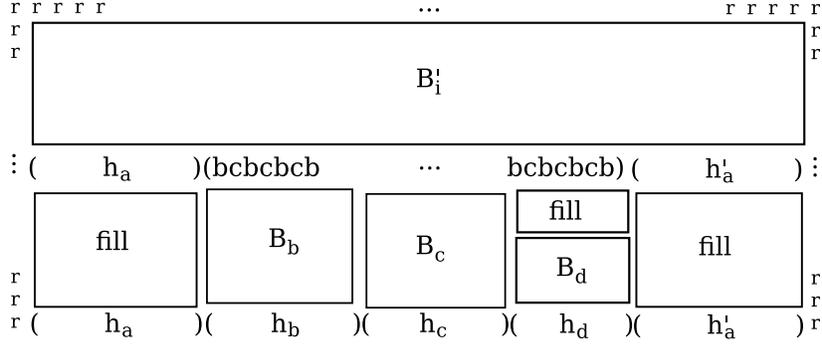}
\caption{Extension of $g_i$}  
\label{extensionfig}
\end{figure}

By Lemma~\ref{extension} (ii), $h_c$ extends to a homomorphism $\hat{h}_c$
of some $B_c = B_{\mbox{s}}(2\ell,\ell)$ to $H$ 
($\ell$ being the length of $p_e(c_i)$), which is identically equal 
to $c_i$ on its left, upper and right sides. 
Similarly $h_b$ and $h_d$ extend to $\hat{h}_b$ and $\hat{h}_d$, 
which are identically equal to $b_i$ and $d_i$ respectively on their 
left, upper and right sides. 
Moreover, the extensions $\hat{h}_b, \hat{h}_c, \hat{h}_d$
of $h_b, h_c, h_d$ to $B_b, B_c, B_d$ can be carried
simultaneously side by side, since $c_i$ is adjacent to $b_i$ and $d_i$.

We then add one level to $B_c$ and extend 
$\hat{h}_c$ by mapping the new level to $b_i$, and
add two levels to each of $B_b$ and $B_d$ and extend $\hat{h}_b$ 
and $\hat{h}_d$ by mapping
the new levels to $c_i$ and $b_i$ respectively.
We then use Lemma~\ref{extension} (i) on the top two levels of $Q_b, Q_c$ and $Q_d$
to add levels and equalize their heights if necessary. We can then
use Lemma~\ref{extension} (i) to extend $h_a$ and $h'_a$ upwards
to match the height of $B_b, B_c$ and $B_d$. 

We have thus extended $g_i$ to 
$\hat{g}_i: B_i \rightarrow H$ such that the restriction $g'_i$ of $\hat{g}_i$
to the top two levels of $Q_i$ satisfies 
$$\omega(g'_i) = \boucle(\gamma_i) \circ \omega(p_e(a_i)) 
\circ (a_i,b_i)((c_i,b_i)^{-1}(c_i,b_i))^{e_i}(a_i,b_i)^{-1}
\circ \omega(p_e(a_i))^{-1} \circ \boucle(\gamma_i)^{-1},$$
for some $e_i$ (see Figure~\ref{extensionfig}).  
Using Lemma~\ref{extension} (ii), we extend $g'_i$ 
to $\hat{g}'_i : B'_{i} \rightarrow H$ which is identically $r$
on the left, top and right sides of $B'_{i}$.

These extensions of $g_i$ to $\hat{g}_i$ and $\hat{g}'_i$ can be carried
simultaneously for $i = 1, \ldots, k$, since they all have the value
$r$ at their common boundaries. We can then use Lemma~\ref{extension} (i)
to equalize heights, and bring up the part of $g$ corresponding to
$(r,s)(r,s)^{-1}$ to the same height. Since the value of the extensions
is identically $r$ at the top level, we can identify all the vertices
of this top level. We have extended $g$ to the apex of some $M_q(C_n)$.

\subsection{Extension towards the base} The extension of $g$ towards the base
of $M_q(C_n)$ is a second extension independent from the extension to the apex.
It again uses Lemma~\ref{extension}, thus we keep the terminology of extending in the
``upper'' direction. The two extensions will afterwards be merged together
by identifying their bottom level.

In $\mathcal{F}_{\mathcal{V}_1}(A(H))$, the word
$\prod_{i=1}^{k} \gamma_i \cdot \rho_i \cdot \gamma_i^{-1}$
of $(A(H) \cup A(H)^{-1})^*$ simplifies to
$$\sigma_1(f) = \prod_{i=0}^{n'-1} (f(2i),f(2i+1))\cdot (f(2i+2),f(2i+1))^{-1}.$$ 
Specifically, this means that there exists a sequence of basic simplifications of 
$\prod_{i=1}^{k} \gamma_i \cdot \rho_i \cdot \gamma_i^{-1}$
eliminating everything but the terms of $\sigma_1(f)$.
In each of these basic simplification, some term $t_i = x$ will cancel out
either with the next term $t_{i+1} = x^{-1}$, or with a further term
$t_{i+2j} = x^{-1}$, where all the intermediate terms $t_{i+1}, \ldots, t_{i+2j-1}$
have been previously simplified. 

In $\prod_{i=1}^{k} \boucle(\gamma_i \cdot \rho_i \cdot \gamma_i^{-1})$, 
a term $t_i$ is replaced by $\boucle(t_i)$. 
Consider the homomorphism corresponding to the word 
$\boucle(t_i) \circ ((r,s)(r,s)^{-1})^{e_i} \circ \boucle(t_i^{-1})$ for
some $e_i$.
By Lemma~\ref{extension} (ii), it extends to a homomorphism
of some $B_i$ to $H$ with value identically $r$ on its left, 
upper and right sides. We can then add a level with value 
identically $s$, and use Lemma~\ref{extension} (i) to bring
up the remainder of the extension of $g$ to the same level.

In this way, we match each step in the simplification of
$\prod_{i=1}^{k} \gamma_i \cdot \rho_i \cdot \gamma_i^{-1}$
to a corresponding step in the extension. We end up with
an extension where the two upper levels correspond to a word
with the terms $\boucle(f(2i),f(2i+1))$ and $\boucle(f(2i+2),f(2i+1))^{-1}$ 
separated by terms of the form $((r,s)(r,s)^{-1})^{e}$.

Now in this word, between $(f(2i),f(2i-1))^{-1}$ and $(f(2i),f(2i+1))$
is a word of the form 
$\omega(p_e(f(2i))^{-1} \circ ((r,s)(r,s)^{-1})^{e} \circ \omega(p_e(f(2i))$.
By Lemma~\ref{extension} (ii), it extends to a homomorphism
of some brick to $H$ with value identically $f(2i)$ on its left, 
upper and right sides. We can then add a level with value 
identically $f(2i-1)$. 
Similarly between $(f(2i),f(2i+1))$ and $(f(2i+2),f(2i+1))^{-1}$
is a word of the form 
$\omega(p_o(f(2i+1))^{-1} \circ ((r,s)(r,s)^{-1})^{e} \circ \omega(p_o(f(2i+1))$.
By Lemma~\ref{extension} (ii), it extends to a homomorphism
of some brick to $H$ with value identically $f(2i+1)$ on its left, 
upper and right sides. We can then add a level with value 
identically $f(2i)$. 

These extensions can be carried out side by side
simultaneously, since $f(j)$ is adjacent to $f(j+1)$ for all $j$.
We use Lemma~\ref{extension} (i) to equalize the height. We now have
extended $g$ so that (after a suitable cyclic shift), the homomorphism 
on the top two levels corresponds to a word of the form
$$\begin{array}{l}
\prod_{i=0}^{n'-1} \left [ ((f(2i),f(2i+1))(f(2i),f(2i+1))^{-1})^{e_i}(f(2i),f(2i+1)) \right. \\
\left. \cdot ((f(2i+2),f(2i+1))^{-1} (f(2i+2),f(2i+1)))^{e_i'} (f(2i+2),f(2i+1))^{-1} \right ].
\end{array}
$$

\subsection{Connection to the base} Let $f': C_n \times K_2 \rightarrow H$
be the homomorphism corresponding to the top two levels of the second extension
of $g$. Note that $f'$ follows the original $f$ twice around its image,
with every arc traced back and forth many times. More precisely,
$f' = h \circ f$, where $h: C_n \times K_2 \rightarrow C_{n'}$
is a homomorphism. The only thing missing is to have $h(i,0) = h(i,1)$ 
for all $i$, that is, to have $f'$ equal on the two levels.
These levels could then be identified to form the base
of $M_q(C_{n})$. 

We label the vertices on $C_n \times K_2 = C_{2n}$ consecutively 
$u_0, \ldots, u_{2n-1}$ (with indices in $\mathbb{Z}_{2n}$). 
We label the edge $[u_i,u_{i+1}]$ with the sign $+$ (resp. $-$)
if for some $j \in \mathbb{Z}_{n'}$ we have $h(u_i) = j$ and $h(u_{i+1}) = j+1$
(resp $h(u_{i+1}) = j-1$). The word $\omega(f')$ above shows that
the number of edges with label $+$ is $n+n'$, and the number of
edges with label $-$ is $n-n'$.

Opposite signs on two consecutive edges $[u_i,u_{i+1}],[u_{i+1},u_{i+2}]$ 
happen precisely when $h(u_i) = h(u_{i+2})$.
We then have $h(u_{i+1}) \in \{h(u_i)-1,h(u_i)+1\}$. Substituting
one value for the other will interchange the signs of 
$[u_i,u_{i+1}]$ and $[u_{i+1},u_{i+2}]$. 
In terms of the extension of $g$, this corresponds to adding two 
levels to match the operation. That is, the extension on the two
new levels is identical to the extension on the previous two levels,
except that value at the vertex corresponding to $u_{i+1}$ switches 
from one term in $\{f(h(u_i)-1),f(h(u_i)+1)\}$ to the other. 

Proceeding this way, we can move the labels around in any way we please.
In particular we can rearrange the labels until there are $(n+n')/2$
``$+$'' labels followed by $(n-n')/2$ ``$-$'' labels, then by $(n+n')/2$ ``$+$'' labels 
and $(n-n')/2$ ``$-$'' labels. In this way, between any $u_i$ and $u_{i+n}$, the value
of $h$ moves forward $(n+n')/2$ times and backward $(n-n')/2$ times
so that it ends up $n'$ places forward in $C_{n'}$, exactly where it started.
Our second extension of $g$ is then equal on the top two levels, so that
these levels can be identified to form the base
of $M_q(C_{n})$. This concludes the proof of Theorem~\ref{coindth}.
\cqfd

\section{Proof of Theorem \ref{indth}} \label{prindth}

\subsection{Algorithmic considerations} 
To prove Theorem~\ref{indth}, we will show that the existence
of an odd cycle in $H$ with zero signature is incompatible
with the existence of a $\mathbb{Z}_2$-map of some
$\basu^m(\hoco{H})$ to $Q_1$. Note that $\mbox{Hom}$ and $\basu$ 
are both exponential constructions. We will first show that the
detection of an odd cycle with zero signature can be done efficiently
in terms of the size of $H$.

To each arc $(u,v)$ of $H$ we associate a variable $X_{u,v} \in \mathbb{Z}_2$.
We consider the system consting of the following equations.
\begin{itemize}
\item The {\em flow constraint} at a vertex
$u$ of $H$ is the equation 
$$\sum_{v \in N_H(u)} (X_{u,v} + X_{v,u}) = 0.$$
(Where $N_H(u)$ is the neighbourhood of $u$ in $H$.)
\item The {\em parity constraint} is the global condition
$$\sum_{(u,v) \in A(H)} X_{u,v} = 1.$$
\item The {\em signature constraint} is the equation
$$\sum_{(u,v) \in A(H)} \left ( (X_{u,v} - X_{v,u})\cdot (u,v) \right ) = 0.$$
\end{itemize}
The flow and parity constraints are equations in $\mathbb{Z}_2$.
The signature constraint is a single equation in $\mathcal{G}_{\mathcal{V}_2}(H)$.
The group $\mathcal{G}_{\mathcal{V}_2}(H)$ has a minimal generating set with 
no more than $|A(H)|$ elements.
Each arc $(u,v)$ can be expressed in terms of this generating set.
The signature constraint is then the set of constraints
corresponding to voiding the coefficient of each element of this generating set
in the expression $\sum_{(u,v) \in A(H)} (X_{u,v} - X_{v,u})\cdot (u,v)$.
Thus the system has no more than $|V(H)| + 1 + |A(H)|$ linear
equations in $|A(H)|$ variables.

If for some odd cycle $C_n$ there is a homomorphism $f: C_n \rightarrow H$
such that $\sigma_2(f) = 0$, then the system above has a solution, obtained by putting
$X_{u,v} = 1$ if there is an odd number of elements $i$ of $\mathbb{Z}_n$
such that $f(i) = u$ and $f(i+1) = v$, and $X_{u,v} = 0$ otherwise.
Conversely, for each solution of the system, the subdigraph of $H$
spanned by the arcs $(u,v)$ such that $X_{u,v} = 1$ is Eulerian,
though not necessarily connected. We can modify the solution by giving
the value $1$ to variables $X_{u,v}$, $X_{v,u}$ corresponding to opposite arcs,
to make the subdigraph connected. Indeed it is clear that this modification
does not alter the validity of the solution.
An Euler tour of the subdigraph then corresponds
to a homomorphism $f: C_n \rightarrow H$ with $n$ odd,
such that $\sigma_2(f) = 0$.
Thus, the existence of some odd cycle with zero signature 
can be detected in polynomial time.

\subsection{$\mathbb{Z}_2$-maps of crown}
Most of our work will involve groups of the form
$\mathbb{Z}_2^{P^{\underline{2}}}$, where $P$ is a poset.
Here, $\underline{2}$ is the poset with elements $0, 1$ such that $0 < 1$.
For a poset $P$, $P^{\underline{2}}$ is the set of all order-preserving maps
of $\underline{2}$ to $P$, including the constant maps.
We will represent an element of $P^{\underline{2}}$ by the comparability
$(x \leq y)$ it represents.
$\mathbb{Z}_2^{P^{\underline{2}}}$ is the 2-group generated by $P^{\underline{2}}$.
To an order preserving-map $f: P \rightarrow Q$, we naturally associate the 
group homomorphism 
$\hat{f}: \mathbb{Z}_2^{P^{\underline{2}}} \rightarrow \mathbb{Z}_2^{Q^{\underline{2}}}$
which extends the map between generators defined by $f$.

A {\em $\mathbb{Z}_2$-crown} is a $\mathbb{Z}_2$-poset $P$ with elements
$\pm 0, \ldots, \pm(2n-1)$ and the relations
$$0 < 1 > 2 < 3 > \cdots < 2n-1 > -0 < -1 > -2 < -3 > \cdots < -(2n-1) > 0.$$
The {\em strict order indicator} $\soi{P}$ on $P$ is the element of 
$\mathbb{Z}_2^{P^{\underline{2}}}$ with value $1$ on injective maps 
and $0$ on constant maps.
Our argument is partly based on the following result, which is
a simplicial statement of the fact that an antipodal continuous self-map
of the circle has ``odd degree''.
\begin{lemma} \label{odddegree}
Let $P$ be a $\mathbb{Z}_2$-crown and $f: P \rightarrow Q_1$
a $\mathbb{Z}_2$-map. Then $\hat{f}(\soi{P}) = \soi{Q_1}$. 
\end{lemma}
\begin{proof}
Each connected component of $f^{-1}(+0)$ and of $f^{-1}(-0)$
starts and ends in an even number, and each connected component 
of $f^{-1}(+1)$ and of $f^{-1}(-1)$ starts and ends in an odd number.
Thus the coefficient of each of $(+0,+0)$, $(-0,-0)$, $(+1,+1)$, $(-1,-1)$
in $\hat{f}(\soi{P})$ is $0$. Now suppose without loss of generality
that $f(0) = 0$ and $f(-0) = -0$. Then on the ``positive zig-zag''
$0 < 1 > 2 < 3 > \cdots < 2n-1 > -0$, the image of $f$ will switch
from $+0$ to $-0$ some $n$ times, and switch back from $-0$
to $+0$ $n-1$ times. Every switch goes through $+1$ or $-1$,
so one of the pairs $(+0,+1), (-0,+1)$ or $(+0,-1), (-0,-1)$
of strict comparibilities is touched an odd number of times, 
and the other an even number of times. Since $f$ is a  $\mathbb{Z}_2$-map,
these numbers are reversed on the ``negative zig-zag''
$-0 < -1 > -2 < -3 > \cdots < -2n+1 > 0$, so that each
of the strict comparibilities $(+0,+1), (-0,+1),(+0,-1), (-0,-1)$
is touched an odd number of times. Thus $\hat{f}(\soi{P}) = \soi{Q_1}$.
\end{proof}

\subsection{Cycles with null signature} A graph
homomorphism $g: C_n \rightarrow H$ induces the $\mathbb{Z}_2$-map
$g': \hoco{C_n} \rightarrow \hoco{H}$ defined by $g'(A,B) = (g(A),g(B))$.
However, it will be useful to associate to $g$ a different map
$g^{+}: \hoco{C} \rightarrow \hoco{H}$ defined as follows.
\begin{itemize}
\item A minimal element $(\{i\}, \{j\})$ of $\hoco{C_n}$ correspond
to an arc $(i,j) = (i, i\pm 1)$ of $C$, and we put
$g^+(\{i\}, \{j\}) = (\{f(i)\}, \{f(j)\})$, its natural image
induced by $g$. 
\item For a maximal element of the form $(\{i\}, \{i-1,i+1\})$,
we put $g^+(\{i\}, \{i-1,i+1\}) = (\{g(i)\},N_H(g(i)))$. Similarly,
for a maximal element of the form $(\{i-1,i+1\}, \{i\})$,
we put $g^+(\{i-1,i+1\}, \{i\}) = (N_H(g(i)),\{g(i)\})$.
\end{itemize}
Thus if $f: \hoco{H} \rightarrow Q_1$
is a $\mathbb{Z}_2$-map, then $\widehat{f\circ g^+}(\soi{\hoco{C}}) = \soi{Q_1}$
by Lemma~\ref{odddegree}. The same holds with barycentric subdivisions:
$\basu^m(\hoco{C})$ is a $\mathbb{Z}_2$-crown,
and if $f: \basu^m(\hoco{H}) \rightarrow Q_1$
is a $\mathbb{Z}_2$-map, then 
$\widehat{f\circ \basu^m(g^+)}(\soi{\basu^m(\hoco{C})}) = \soi{Q_1}$.

Now, if $\sigma_2(g) = 0$, then in $\mathcal{F}_{\mathcal{V}_2}(A(H))$
we have 
\begin{equation}\label{sig2eq}
\sigma_2(g) = \sum_{i=0}^{n-1} 
\left [ (g(2i),g(2i+1)) + (g(2i+2),g(2i+1)) \right ]
= \sum_{j=1}^{k} \rho_j,
\end{equation}
where $\rho_1, \ldots, \rho_k$ are relations
of the form $\rho_j = (a_j,b_j)+(c_j,b_j)+(c_j,d_j)+(a_j,d_j)$
defining the congruence $\theta$ on $\mathcal{F}_{\mathcal{V}_2}(A(H))$.
To $\rho = (a,b)+(c,b)+(c,d)+(a,d)$, we associate the subposet
$\rho^+$ of $\hoco{H}$ induced by the the set $\{
(\{a\},N_H(a))$, $(N_H(b),\{b\})$, $(\{c\},N_H(c))$, $(N_H(d),\{d\})$, 
$(\{a\},\{b\})$, $(\{c\},\{b\})$, $(\{c\},\{d\})$, $(\{a\},\{d\})\}$.
(See Figure~\ref{rhoplus}.)

Our next Lemma adapts Equation~(\ref{sig2eq}) to the groups
$\mathbb{Z}_2^{\hoco{H}^{\underline{2}}}$ and 
$\mathbb{Z}_2^{\hoco{\basu^m(H)}^{\underline{2}}}$ for all $m \geq 1$.
Recall that $g^+: \hoco{C_n} \rightarrow \hoco{H}$ induces
$\widehat{g^+}:\mathbb{Z}_2^{\hoco{C_n}^{\underline{2}}} \rightarrow
\mathbb{Z}_2^{\hoco{H}^{\underline{2}}}$; $g^+$
also induces $\basu^m(g^+): \basu^{m}(\hoco{C_n}) \rightarrow 
\basu^m(\hoco{H})$ for all $m \geq 1$, which in turn induce
$\widehat{\basu^m(g^+)}: \mathbb{Z}_2^{\basu^m(\hoco{C_n})^{\underline{2}}} 
\rightarrow \mathbb{Z}_2^{\basu^m(\hoco{H})^{\underline{2}}}$.
\begin{lemma} \label{sigmasum}
$$\widehat{g^+}(\soi{\hoco{C_n}}) = \sum_{j=1}^{k} \soi{\rho_j^+},$$
hence
$$\widehat{\basu^m(g)}(\soi{\basu^m(\hoco{C_n})}) = \sum_{j=1}^{k} \soi{\basu^m(\rho_j^+)}$$
for all $m \geq 1$.
\end{lemma}
\begin{proof}
Let $O_g$ be the set of arcs of $H$ which appear
an odd number of times as $(g(i),g(i+1))$ or $(g(i+1),g(i))$
for some $i \in \mathbb{Z}_n$. Thus in $\mathcal{F}_{\mathcal{V}_2}(A(H))$,
we have $\sigma_2(g) = \sum_{(u,v) \in O_g} (u,v)$.
By Equation~\ref{sig2eq}, $O_g$ coincides with the set of
arcs which appear an odd number of times as terms
in $\rho_1, \ldots, \rho_k$.

Now for $(u,v) \in A(H)$, let $V(u,v)$ be the subposet 
$$(N_H(v),\{v\}) > (\{u\},\{v\}) < (\{u\},N_H(u))$$ 
of $\hoco{H}$. Then 
$$\widehat{g^+}(\soi{\hoco{C_n}}) = \sum_{(u,v) \in O_g} \soi{V(u,v)},$$
and similarly for all
$$\widehat{\basu^m(g)}(\soi{\basu^m(\hoco{C})})
 = \sum_{(u,v) \in O_g} \soi{\basu^m(V(u,v))}$$
for all $m \geq 1$.
Equation~\ref{sig2eq} then implies that 
$$\sum_{(u,v) \in O_g} \soi{V(u,v)} = \sum_{j=1}^{k} \soi{\rho_j^+}$$
and
$$\sum_{(u,v) \in O_g} \soi{\basu^m(V(u,v))}
= \sum_{j=1}^{k} \soi{\basu^m(\rho_j^+)}$$
for all $m \geq 1$.
\end{proof}
We will next see that the conclusion 
$$\widehat{\basu^m(g)}(\soi{\basu^m(\hoco{C})}) = \sum_{j=1}^{k} \soi{\basu^m(\rho_j^+)}$$
of Lemma~\ref{sigmasum} is incompatible with the conclusion
$$\widehat{\basu^m(g)}(\soi{\basu^m(\hoco{C})})
= \soi{Q_1}$$ of Lemma~\ref{odddegree}.

\subsection{Domination and dismantlability}
In a poset $Q$, an element $p$ is said to be {\em dominated} by an element $q$
if $p$ is comparable to $q$ and every element comparable to $p$ is
comparable in the same way to $q$. $Q$ is said to be {\em dismantlable}
if it can be reduced to a single point by recursively removing dominated elements.

\begin{figure}[h]
\centering
\includegraphics{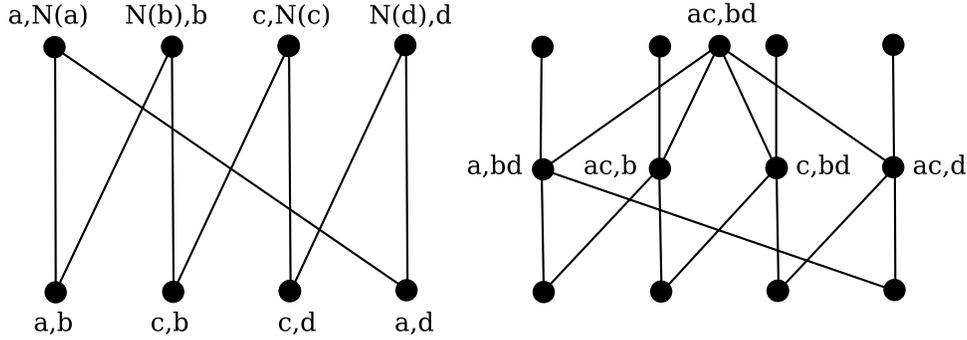}
\caption{$\rho^+$ and $D(\rho^+)$}  
\label{rhoplus}
\end{figure}

For instance, consider the subposet $D(\rho^+)$ of $\hoco{H}$ obtained by adding
the elements $(\{a\},\{b,d\})$, $(\{a,c\},\{b\})$, $(\{c\},\{b,d\})$,
$(\{a,c\},\{d\})$ and $(\{a,c\},\{b,d\})$ to $\rho^+$. (See Figure~\ref{rhoplus}.)
In $D(\rho^+)$, $(\{a\},N_H(a))$, $(N_H(b),\{b\})$, $(\{c\},N_H(c))$
and $(N_H(d),\{d\})$ are dominated respectively by
$(\{a\},\{b,d\})$, $(\{a,c\},\{b\})$, $(\{c\},\{b,d\})$ and
$(\{a,c\},\{d\})$. Removing these dominated elements leaves
$(\{a,c\},\{b,d\})$ as the unique maximum. We then have 
$(\{a\},\{b,d\})$, $(\{a,c\},\{b\})$, $(\{c\},\{b,d\})$ and $(\{a,c\},\{d\})$
dominated by $(\{a,c\},\{b,d\})$, and removing these leaves
$(\{a\},\{b\})$, $(\{c\},\{b\})$, $(\{c\},\{b\})$ and $(\{c\},\{d\})$
dominated by $(\{a,c\},\{b,d\})$. Hence $D(\rho^+)$
is dismantlable.

\begin{lemma} \label{diszero}
Let $P$ be a $\mathbb{Z}_2$-crown, $Q$ a dismantlable poset
and $f: P \rightarrow Q_1$ an order-preserving map which 
factors through $Q$,
that is, $f = h \circ g$ where $g : P \rightarrow Q$,
$h: Q \rightarrow Q_1$ are order-preserving.
Then $\hat{f}(\soi{P}) = 0$.
\end{lemma}
\begin{proof} The result is clear if $Q$ is a single point. Thus
we can proceed by induction on the number of elements in $Q$.
Let $p$ be dominated by $q$ in $Q$; we will suppose that $p < q$
(the other case being symmetric). Let $r: Q \rightarrow Q$ be the
retraction which maps $p$ to $q$ and fixes everything else.
Then $f' = h \circ r \circ g : P \rightarrow Q_1$ factors through
the dismantlable poset $r(Q)$ which has one element less than $Q$,
so by the induction hypothesis, $\widehat{f'}(\soi{P}) = 0$.
Thus if $f = f'$, then $\hat{f}(\soi{P}) = 0$. We can therefore suppose
that $f \neq f'$. This means that $h(p) \neq h(q)$. We will suppose
without loss of generality that $h(p) = 0$ and $h(q) = 1$. 

Let $g': P \rightarrow Q_1$ be the map obtained from $g$ by changing 
the image of every maximal element $x$ such that $g(x) = p$
from $p$ to $q$. For each such $x$, there are two minimal elements 
$y, z$ of $P$ which are below $x$. We then have $f(y), f(z) \leq f(x) = 0$, 
so that $f(y) = f(z) = 0$. Therefore the two compararabilities
$(y \leq x)$, $(z \leq x)$ are mapped to $(0 \leq 0)$ by $\hat{f}$ 
and to $(0 \leq 1)$ by $\widehat{h \circ g'}$. 
Thus $\widehat{h \circ g'}(\soi{P}) = \hat{f}(\soi{P})$.

Now $r \circ g$ is obtained from $g'$ by changing 
the image of every minimal element $x$ such that $g'(x) = p$
from $0$ to $1$. For each such $x$, there are two maximal elements 
$y, z$ of $P$ which are above $x$. Now  since $q$ dominates $p$,
we have $g'(y), g'(z) \geq q$ hence $h \circ g'(y) = h \circ g'(z) = 1$.
Therefore the two compararabilities
$(x \leq y)$, $(x \leq z)$ are mapped to $(0 \leq 1)$ by $\widehat{h \circ g'}$
and to $(1 \leq 1)$ by $\widehat{h \circ r \circ g} = \widehat{f'}$.
Thus $\widehat{f'}(\soi{P}) = \widehat{h \circ g'}(\soi{P}) = \hat{f}(\soi{P})$.
Therefore $\widehat{f'}(\soi{P}) = 0$ implies $\hat{f}(\soi{P}) = 0$. 
\end{proof}

\begin{lemma} \label{disbasu} If $Q$ is a dismantlable poset, then 
for any $m$, $\basu^m(Q)$ is dismantlable.
\end{lemma}
\begin{proof} It suffices to show that if $Q$ is dismantlable,
then $\basu(Q)$ is dismantlable. We will again proceed by induction 
on the number of elements in $Q$, the result being clear if
$Q$ is a single point. 
Let $p$ be dominated by $q$ in $Q$. We will show that
$\basu(Q)$ dismantles to $\basu(Q \setminus \{p\})$.
The elements of $\basu(Q)$ are chains in $Q$,
and since $q$ dominates $p$, for every element $C$
of $\basu(Q)$ containing $p$, $C \cup \{q\}$ is an
element of $\basu(Q)$. Let $m$ be the number of elements
of $\basu(Q)$ which contain $p$ but not $q$.
We construct
a sequence $\basu(Q) = R_0, R_1, \ldots, R_m$ of subposets of $\basu(Q)$, 
where $R_i$ is obtained from $R_{i-1}$ by removing a maximal element $C_i$
of $R_{i-1}$ which contains $p$ but not $q$. Since $C_i$ is dominated by
$C_i \cup \{q\}$ in $R_{i-1}$, the sequence is a dismantling of
$\basu(Q)$ to its subposet $R_m$ which consists of all the elements
which contain $q$ whenever they contain $p$.
Let $R_m, R_{m+1}, \ldots, R_{2m}$ be a sequence of subposets of 
$R_m$, where $R_i$ is obtained from $R_{i-1}$ by removing a minimal element $C_i$
of $R_{i-1}$ which contains $p$. Since $C_i$ is dominated by
$C_i \setminus \{p\}$ in $R_{i-1}$, the sequence is a dismantling of
$R_m$ to its subposet $R_{2m} = \basu(Q \setminus \{p\})$.
Thus $\basu(Q)$ dismantles to $\basu(Q \setminus \{p\})$. Therefore
if $\basu(Q \setminus \{p\})$ is dismantlable, then so is $\basu(Q)$.
\end{proof}

\begin{proof}[Proof of Theorem~\ref{indth}]
Suppose that for some odd $n$ there exists a homomorphism 
$f: C_n \rightarrow H$ such that $\sigma_2(g) = 0$. 
Then there exists a sequence $\rho_1, \ldots, \rho_k$ of generators of
$\theta$ such that $\sigma_2(f) = \sum_{j=1}^{k} \rho_j$
in $\mathcal{F}_{\mathcal{V}_2}(A(H))$.
By Lemma~\ref{sigmasum}, we then have 
$$\widehat{\basu^m(f^+)}(\soi{\basu^m(\hoco{C_n})}) = \sum_{j=1}^{k} \soi{\basu^m(\rho_j^+)}$$
for all $m \geq 0$. Now for $j = 1, \ldots, k$, $\rho_j^+$ is contained in $D(\rho_j^+)$
which is dismantlable, hence $\basu^m(D(\rho_j^+))$ is dismantlable for all $m$
by Lemma~\ref{disbasu}. Hence by Lemma~\ref{disbasu},
for any order-preserving map 
$g: \basu^m(\hoco{H}) \rightarrow Q_1$, we have 
$\hat{g}(\soi{\basu^m(\rho_j^+)}) = 0$ for
$j = 1, \ldots, k$. Therefore
$$\widehat{g \circ \basu^m(f^+)}(\soi{\basu^m(\hoco{C_n})}) 
= \sum_{j=1}^{k} \hat{g}(\soi{\basu^m(\rho_j^+)}) = 0.$$
By Lemma~\ref{odddegree}, 
$g \circ \basu^m(f^+): \basu^m(\hoco{C_n}) \rightarrow Q_1$
cannot be a $\mathbb{Z}_2$-map. Since 
$\basu^m(f^+): \basu^m(\hoco{C_n}) \rightarrow \basu^m(\hoco{H})$
is a  $\mathbb{Z}_2$-map, we conclude that there does not exist
a $\mathbb{Z}_2$-map $g: \basu^m(\hoco{H}) \rightarrow Q_1$.
Therefore $\ind(\hoco{H}) \geq 2$.
\end{proof}

\end{document}